\documentclass{article}

\usepackage{amssymb,amsmath,amsthm}
\usepackage[]{algorithm2e}

\usepackage{pgf,tikz}
\usepackage{multicol}

\newtheorem{definition}{Definition}[section]
\newtheorem{theorem}[definition]{Theorem}

\newtheorem{remark}[definition]{Remark}
\newtheorem{lemma}[definition]{Lemma}
\newtheorem{proposition}[definition]{Proposition}

\usepackage{ifxetex}
\ifxetex
  \usepackage{fontspec}
\else
  \usepackage[T1]{fontenc}
  \usepackage[utf8]{inputenc}
  \usepackage{lmodern}
\fi

\title{Domination in the Sierpi\'{n}ski graphs $S(K_n,t)$}
\date{August 22, 2020}

\author{Chia-An Liu
\thanks{Department of Mathematics, Soochow University, Taipei 111002, Taiwan R.O.C.
{\tt Email: liuchiaan8@gmail.com}}}

\begin{document}
\maketitle

\begin{abstract}
Different types of domination on the Sierpi\'{n}ski graphs $S(K_n,t)$ will be studied in this paper. More precisely, we propose a minimal dominating set for $S(K_n,t)$ so that the exact values of their domination numbers, Roman domination numbers, and double Roman domination numbers are given. As applications, some previous bounds and results are confirmed to be tight and further generalized.
\end{abstract}

\bigskip


\smallskip


\medskip

\section{Introduction and preliminary}      \label{sec_intro}

Let $[n]=\{1,2,\ldots,n\}$ be the set of positive integers at most $n.$ For every pair of positive integers $n$ and $t,$ the Sierpi\'{n}ski graph $S(K_n,t)$ is defined as the simple graph with vertices set $[n]^t=\{v_1v_2\ldots v_t\mid v_i\in[n]~\text{for}~1\leq i\leq t\}$, in which $u_1u_2\ldots u_t$ and $v_1v_2\ldots v_t$ are adjacent if and only if there exists $s\in[t]$ satisfying
\begin{equation}
\left\{\begin{array}{ll}
u_j=v_j & \text{if}~j<s;
\\
u_s\neq v_s; &
\\
u_j=v_s~\text{and}~v_j=u_s & \text{if}~j>s.
\end{array}\right.
\nonumber
\end{equation}
In short, the consecutively repeated entries in a vertex are often written together. For example, the vertex $u_1u_2\underbrace{u_3\ldots u_3}_{t-2}$ can be denoted by $u_1u_2u_3^{t-2}.$ See Figure 1 as an example of $S(K_n,t)$ when $n=4$ and $t=3.$

\medskip

\begin{center}
\begin{tikzpicture}
\draw[] (0,7)coordinate(v111) circle(0.1cm);
\draw[] (-0.3,7.2)coordinate(111) node{\footnotesize 111};
\draw[] (0,6)coordinate(v112) circle(0.1cm);
\draw[] (-0.4,6)coordinate(112) node{\footnotesize 112};
\draw[] (0,5)coordinate(v121) circle(0.1cm);
\draw[] (-0.4,5)coordinate(121) node{\footnotesize 121};
\draw[] (0,4)coordinate(v122) circle(0.1cm);
\draw[] (-0.4,4)coordinate(122) node{\footnotesize 122};
\draw[] (0,3)coordinate(v211) circle(0.1cm);
\draw[] (-0.4,3)coordinate(211) node{\footnotesize 211};
\draw[] (0,2)coordinate(v212) circle(0.1cm);
\draw[] (-0.4,2)coordinate(212) node{\footnotesize 212};
\draw[] (0,1)coordinate(v221) circle(0.1cm);
\draw[] (-0.4,1)coordinate(221) node{\footnotesize 221};
\draw[] (0,0)coordinate(v222) circle(0.1cm);
\draw[] (-0.3,-0.2)coordinate(222) node{\footnotesize 222};
\draw[] (1,7)coordinate(v114) circle(0.1cm);
\draw[] (1,7.25)coordinate(114) node{\footnotesize 114};
\draw[] (1,6)coordinate(v113) circle(0.1cm);
\draw[] (0.9,5.7)coordinate(113) node{\footnotesize 113};
\draw[] (1,5)coordinate(v124) circle(0.1cm);
\draw[] (0.9,5.3)coordinate(124) node{\footnotesize 124};
\draw[] (1,4)coordinate(v123) circle(0.1cm);
\draw[] (1,3.7)coordinate(123) node{\footnotesize 123};
\draw[] (1,3)coordinate(v214) circle(0.1cm);
\draw[] (1,3.3)coordinate(214) node{\footnotesize 214};
\draw[] (1,2)coordinate(v213) circle(0.1cm);
\draw[] (0.9,1.7)coordinate(213) node{\footnotesize 213};
\draw[] (1,1)coordinate(v224) circle(0.1cm);
\draw[] (0.9,1.3)coordinate(224) node{\footnotesize 224};
\draw[] (1,0)coordinate(v223) circle(0.1cm);
\draw[] (1,-0.25)coordinate(223) node{\footnotesize 223};
\draw[] (2,7)coordinate(v141) circle(0.1cm);
\draw[] (2,7.25)coordinate(141) node{\footnotesize 141};
\draw[] (2,6)coordinate(v142) circle(0.1cm);
\draw[] (2.1,5.7)coordinate(142) node{\footnotesize 142};
\draw[] (2,5)coordinate(v131) circle(0.1cm);
\draw[] (2.1,5.3)coordinate(131) node{\footnotesize 131};
\draw[] (2,4)coordinate(v132) circle(0.1cm);
\draw[] (2,3.7)coordinate(132) node{\footnotesize 132};
\draw[] (2,3)coordinate(v241) circle(0.1cm);
\draw[] (2,3.3)coordinate(241) node{\footnotesize 241};
\draw[] (2,2)coordinate(v242) circle(0.1cm);
\draw[] (2.1,1.7)coordinate(242) node{\footnotesize 242};
\draw[] (2,1)coordinate(v231) circle(0.1cm);
\draw[] (2.1,1.3)coordinate(231) node{\footnotesize 231};
\draw[] (2,0)coordinate(v232) circle(0.1cm);
\draw[] (2,-0.25)coordinate(232) node{\footnotesize 232};
\draw[] (3,7)coordinate(v144) circle(0.1cm);
\draw[] (3,7.25)coordinate(144) node{\footnotesize 144};
\draw[] (3,6)coordinate(v143) circle(0.1cm);
\draw[] (3.3,6.2)coordinate(143) node{\footnotesize 143};
\draw[] (3,5)coordinate(v134) circle(0.1cm);
\draw[] (3.3,5.2)coordinate(134) node{\footnotesize 134};
\draw[] (3,4)coordinate(v133) circle(0.1cm);
\draw[] (2.9,3.7)coordinate(133) node{\footnotesize 133};
\draw[] (3,3)coordinate(v244) circle(0.1cm);
\draw[] (2.9,3.3)coordinate(244) node{\footnotesize 244};
\draw[] (3,2)coordinate(v243) circle(0.1cm);
\draw[] (3.3,2.2)coordinate(243) node{\footnotesize 243};
\draw[] (3,1)coordinate(v234) circle(0.1cm);
\draw[] (3.3,1.2)coordinate(234) node{\footnotesize 234};
\draw[] (3,0)coordinate(v233) circle(0.1cm);
\draw[] (3,-0.25)coordinate(233) node{\footnotesize 233};
\draw[] (4,7)coordinate(v411) circle(0.1cm);
\draw[] (4,7.25)coordinate(411) node{\footnotesize 411};
\draw[] (4,6)coordinate(v412) circle(0.1cm);
\draw[] (3.7,5.8)coordinate(412) node{\footnotesize 412};
\draw[] (4,5)coordinate(v421) circle(0.1cm);
\draw[] (3.7,4.8)coordinate(421) node{\footnotesize 421};
\draw[] (4,4)coordinate(v422) circle(0.1cm);
\draw[] (4.1,3.7)coordinate(422) node{\footnotesize 422};
\draw[] (4,3)coordinate(v311) circle(0.1cm);
\draw[] (4.1,3.3)coordinate(311) node{\footnotesize 311};
\draw[] (4,2)coordinate(v312) circle(0.1cm);
\draw[] (3.7,1.8)coordinate(312) node{\footnotesize 312};
\draw[] (4,1)coordinate(v321) circle(0.1cm);
\draw[] (3.7,0.8)coordinate(321) node{\footnotesize 321};
\draw[] (4,0)coordinate(v322) circle(0.1cm);
\draw[] (4,-0.25)coordinate(322) node{\footnotesize 322};
\draw[] (5,7)coordinate(v414) circle(0.1cm);
\draw[] (5,7.25)coordinate(414) node{\footnotesize 414};
\draw[] (5,6)coordinate(v413) circle(0.1cm);
\draw[] (4.9,5.7)coordinate(413) node{\footnotesize 413};
\draw[] (5,5)coordinate(v424) circle(0.1cm);
\draw[] (4.9,5.3)coordinate(424) node{\footnotesize 424};
\draw[] (5,4)coordinate(v423) circle(0.1cm);
\draw[] (5,3.7)coordinate(423) node{\footnotesize 423};
\draw[] (5,3)coordinate(v314) circle(0.1cm);
\draw[] (5,3.3)coordinate(314) node{\footnotesize 314};
\draw[] (5,2)coordinate(v313) circle(0.1cm);
\draw[] (4.9,1.7)coordinate(313) node{\footnotesize 313};
\draw[] (5,1)coordinate(v324) circle(0.1cm);
\draw[] (4.9,1.3)coordinate(324) node{\footnotesize 324};
\draw[] (5,0)coordinate(v323) circle(0.1cm);
\draw[] (5,-0.25)coordinate(323) node{\footnotesize 323};
\draw[] (6,7)coordinate(v441) circle(0.1cm);
\draw[] (6,7.25)coordinate(441) node{\footnotesize 441};
\draw[] (6,6)coordinate(v442) circle(0.1cm);
\draw[] (6.1,5.7)coordinate(442) node{\footnotesize 442};
\draw[] (6,5)coordinate(v431) circle(0.1cm);
\draw[] (6.1,5.3)coordinate(431) node{\footnotesize 431};
\draw[] (6,4)coordinate(v432) circle(0.1cm);
\draw[] (6,3.7)coordinate(432) node{\footnotesize 432};
\draw[] (6,3)coordinate(v341) circle(0.1cm);
\draw[] (6,3.3)coordinate(341) node{\footnotesize 341};
\draw[] (6,2)coordinate(v342) circle(0.1cm);
\draw[] (6.1,1.7)coordinate(342) node{\footnotesize 342};
\draw[] (6,1)coordinate(v331) circle(0.1cm);
\draw[] (6.1,1.3)coordinate(331) node{\footnotesize 331};
\draw[] (6,0)coordinate(v332) circle(0.1cm);
\draw[] (6,-0.25)coordinate(332) node{\footnotesize 332};
\draw[] (7,7)coordinate(v444) circle(0.1cm);
\draw[] (7.3,7.2)coordinate(444) node{\footnotesize 444};
\draw[] (7,6)coordinate(v443) circle(0.1cm);
\draw[] (7.4,6)coordinate(443) node{\footnotesize 443};
\draw[] (7,5)coordinate(v434) circle(0.1cm);
\draw[] (7.4,5)coordinate(434) node{\footnotesize 434};
\draw[] (7,4)coordinate(v433) circle(0.1cm);
\draw[] (7.4,4)coordinate(433) node{\footnotesize 433};
\draw[] (7,3)coordinate(v344) circle(0.1cm);
\draw[] (7.4,3)coordinate(344) node{\footnotesize 344};
\draw[] (7,2)coordinate(v343) circle(0.1cm);
\draw[] (7.4,2)coordinate(343) node{\footnotesize 343};
\draw[] (7,1)coordinate(v334) circle(0.1cm);
\draw[] (7.4,1)coordinate(334) node{\footnotesize 334};
\draw[] (7,0)coordinate(v333) circle(0.1cm);
\draw[] (7.3,-0.2)coordinate(333) node{\footnotesize 333};
\draw(v111)--(v222)--(v333)--(v444)--(v111)--(v333);
\draw(v222)--(v444);
\draw(v122)--(v144);
\draw(v322)--(v344);
\draw(v411)--(v433);
\draw(v211)--(v233);
\draw(v122)--(v133);
\draw(v422)--(v433);
\draw(v211)--(v244);
\draw(v311)--(v344);
\draw(v144)--(v133);
\draw(v411)--(v422);
\draw(v244)--(v233);
\draw(v311)--(v322);
\draw(v112)--(v113)--(v114)--(v112);
\draw(v141)--(v142)--(v143)--(v141);
\draw(v121)--(v123)--(v124)--(v121);
\draw(v131)--(v132)--(v134)--(v131);
\draw(v212)--(v213)--(v214)--(v212);
\draw(v241)--(v242)--(v243)--(v241);
\draw(v221)--(v223)--(v224)--(v221);
\draw(v231)--(v232)--(v234)--(v231);
\draw(v312)--(v313)--(v314)--(v312);
\draw(v341)--(v342)--(v343)--(v341);
\draw(v321)--(v323)--(v324)--(v321);
\draw(v331)--(v332)--(v334)--(v331);
\draw(v412)--(v413)--(v414)--(v412);
\draw(v441)--(v442)--(v443)--(v441);
\draw(v421)--(v423)--(v424)--(v421);
\draw(v431)--(v432)--(v434)--(v431);

\draw[fill=white] (0,7)coordinate(v111) circle(0.1cm);
\draw[fill=white] (0,6)coordinate(v112) circle(0.1cm);
\draw[fill=white] (0,5)coordinate(v121) circle(0.1cm);
\draw[fill=white] (0,4)coordinate(v122) circle(0.1cm);
\draw[fill=white] (0,3)coordinate(v211) circle(0.1cm);
\draw[fill=white] (0,2)coordinate(v212) circle(0.1cm);
\draw[fill=white] (0,1)coordinate(v221) circle(0.1cm);
\draw[fill=white] (0,0)coordinate(v222) circle(0.1cm);
\draw[fill=white] (1,7)coordinate(v114) circle(0.1cm);
\draw[fill=white] (1,6)coordinate(v113) circle(0.1cm);
\draw[fill=white] (1,5)coordinate(v124) circle(0.1cm);
\draw[fill=white] (1,4)coordinate(v123) circle(0.1cm);
\draw[fill=white] (1,3)coordinate(v214) circle(0.1cm);
\draw[fill=white] (1,2)coordinate(v213) circle(0.1cm);
\draw[fill=white] (1,1)coordinate(v224) circle(0.1cm);
\draw[fill=white] (1,0)coordinate(v223) circle(0.1cm);
\draw[fill=white] (2,7)coordinate(v141) circle(0.1cm);
\draw[fill=white] (2,6)coordinate(v142) circle(0.1cm);
\draw[fill=white] (2,5)coordinate(v131) circle(0.1cm);
\draw[fill=white] (2,4)coordinate(v132) circle(0.1cm);
\draw[fill=white] (2,3)coordinate(v241) circle(0.1cm);
\draw[fill=white] (2,2)coordinate(v242) circle(0.1cm);
\draw[fill=white] (2,1)coordinate(v231) circle(0.1cm);
\draw[fill=white] (2,0)coordinate(v232) circle(0.1cm);
\draw[fill=white] (3,7)coordinate(v144) circle(0.1cm);
\draw[fill=white] (3,6)coordinate(v143) circle(0.1cm);
\draw[fill=white] (3,5)coordinate(v134) circle(0.1cm);
\draw[fill=white] (3,4)coordinate(v133) circle(0.1cm);
\draw[fill=white] (3,3)coordinate(v244) circle(0.1cm);
\draw[fill=white] (3,2)coordinate(v243) circle(0.1cm);
\draw[fill=white] (3,1)coordinate(v234) circle(0.1cm);
\draw[fill=white] (3,0)coordinate(v233) circle(0.1cm);
\draw[fill=white] (4,7)coordinate(v411) circle(0.1cm);
\draw[fill=white] (4,6)coordinate(v412) circle(0.1cm);
\draw[fill=white] (4,5)coordinate(v421) circle(0.1cm);
\draw[fill=white] (4,4)coordinate(v422) circle(0.1cm);
\draw[fill=white] (4,3)coordinate(v311) circle(0.1cm);
\draw[fill=white] (4,2)coordinate(v312) circle(0.1cm);
\draw[fill=white] (4,1)coordinate(v321) circle(0.1cm);
\draw[fill=white] (4,0)coordinate(v322) circle(0.1cm);
\draw[fill=white] (5,7)coordinate(v414) circle(0.1cm);
\draw[fill=white] (5,6)coordinate(v413) circle(0.1cm);
\draw[fill=white] (5,5)coordinate(v424) circle(0.1cm);
\draw[fill=white] (5,4)coordinate(v423) circle(0.1cm);
\draw[fill=white] (5,3)coordinate(v314) circle(0.1cm);
\draw[fill=white] (5,2)coordinate(v313) circle(0.1cm);
\draw[fill=white] (5,1)coordinate(v324) circle(0.1cm);
\draw[fill=white] (5,0)coordinate(v323) circle(0.1cm);
\draw[fill=white] (6,7)coordinate(v441) circle(0.1cm);
\draw[fill=white] (6,6)coordinate(v442) circle(0.1cm);
\draw[fill=white] (6,5)coordinate(v431) circle(0.1cm);
\draw[fill=white] (6,4)coordinate(v432) circle(0.1cm);
\draw[fill=white] (6,3)coordinate(v341) circle(0.1cm);
\draw[fill=white] (6,2)coordinate(v342) circle(0.1cm);
\draw[fill=white] (6,1)coordinate(v331) circle(0.1cm);
\draw[fill=white] (6,0)coordinate(v332) circle(0.1cm);
\draw[fill=white] (7,7)coordinate(v444) circle(0.1cm);
\draw[fill=white] (7,6)coordinate(v443) circle(0.1cm);
\draw[fill=white] (7,5)coordinate(v434) circle(0.1cm);
\draw[fill=white] (7,4)coordinate(v433) circle(0.1cm);
\draw[fill=white] (7,3)coordinate(v344) circle(0.1cm);
\draw[fill=white] (7,2)coordinate(v343) circle(0.1cm);
\draw[fill=white] (7,1)coordinate(v334) circle(0.1cm);
\draw[fill=white] (7,0)coordinate(v333) circle(0.1cm);
\end{tikzpicture}

\bigskip

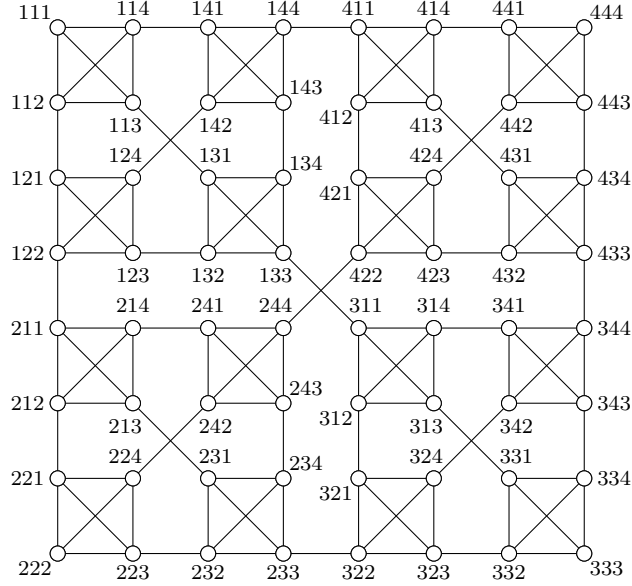
{\bf Figure 1:} The Sierpi\'{n}ski graph $S(K_4,3).$
\end{center}

\medskip

Let $G=(V,E)$ be a graph with vertex set $V$ and edge set $E.$ For each $v\in V,$ $N_G(v)$ denotes the set of vertices adjacent to $v$ in $G$, and $N_G[v]=N_G(v)\cup \{v\}.$ A set $D\subseteq V$ is said to be {\it dominating} in $G$ if $\cup_{v\in D}N_G[v]=V.$ The {\it domination number} $\gamma(G)$ of $G$ is the minimum cardinality among all dominating sets of $G.$ It is well-known, for example, see~\cite{gj:79}, that testing whether $\gamma(G)\leq k$ or not for some input $k$ is an NP-complete problem. A {\it Roman dominating function} on $G$ is defined as a function $f:V\rightarrow\{0,1,2\}$ such that every vertex $u\in V$ with $f(u)=0$ has at least a neighbor $v\in N_G(u)$ satisfying $f(v)=2.$ The weight of $f$ is realized as $f(V)=\sum_{v\in V} f(v),$ and the {\it Roman domination number} of $G$, denoted by $\gamma_R(G),$ is the minimum weight among all Roman dominating functions of $G$. A {\it double Roman dominating function} on $G$ is defined as a function $f:V\rightarrow\{0,1,2,3\}$ such that
\begin{enumerate}
\item[(i)] every vertex $u\in V$ with $f(u)=0$ has at least one neighbor $v\in N_G(u)$ satisfying $f(v)=3$ or at least two distinct neighbors $w,x\in N_G(u)$ satisfying $f(w)=f(x)=2,$ and
\item[(ii)] every vertex $u\in V$ with $f(u)=1$ has at least a neighbor $v\in N_G(u)$ satisfying $f(v)\geq 3.$
\end{enumerate}
Similarly, the weight of $f$ is realized as $f(V)=\sum_{v\in V} f(v),$ and the {\it double Roman domination number} of $G$, denoted by $\gamma_{dR}(G),$ is the minimum weight among all double Roman dominating functions of $G$.

\medskip

Klav\v{z}ar and U. Milutinovi\'{c} introduced the graph $S(K_n,t)$ in~\cite{km:97}, and noticed that as $n=3$ those graphs are exactly the Tower of Hanoi graphs. Later in~\cite{kmp:02}, $S(K_n,t)$ have been called Sierpi\'{n}ski graphs and studied from many aspects. The concept of Sierpi\'{n}ski graphs was generalized so that $S(G,t)$ can be constructed for every simple graph $G$. One can refer to~\cite{gkmmp:13,hh:14,hp:12,kpz:13,rre:17,rrr:17} for more details. The Roman domination was defined in~\cite{rr:20,s:19} and has been studied by many authors. The authors in~\cite{ckpw:09,cdhh:04} proposed inspiring properties and problems involved with Roman domination in graphs. In 2016, Beeler {\it et al.} pioneered the study of double Roman domination in~\cite{bhh:16}. The decision of double Roman domination numbers was verified to be NP-complete for some families of graphs in~\cite{ahsy:16}. Some upper and lower bounds for $\gamma_{dR}(G)$ were given in~\cite{aa:18,aa:20,v:18} in terms of the number of vertices and various parameters in a graph.

\medskip

The following result given in~\cite{bhh:16} will be useful in this paper.
\begin{proposition}     \label{prop_noV1}
In a double Roman dominating function of weight $\gamma_{dR}(G),$ no vertex needs to be assigned the values $1.$
\end{proposition}
\noindent
In~\cite{rrr:17}, Ramezani {\it et al.} made a progress in Roman domination numbers of Sierpi\'{n}ski graphs.
\begin{theorem}     \label{thm_R_leq}
For any integers $n\geq 2$ and $t\geq 1,$
\begin{equation}
\gamma_R(S(K_n,t))\leq
\left\{\begin{array}{ll}
\frac{2n^t+2}{n+1} & \text{if $t$ is odd};  \\
\frac{2n^t+n-1}{n+1} & \text{if $t$ is even}.
\end{array}\right.
\nonumber
\end{equation}
\end{theorem}
\noindent
When $t=2,$ Theorem~\ref{thm_R_leq} was verified to be tight in~\cite{aa:20}. Moreover, the authors in~\cite{aa:20} also gave exact values of $\gamma_{dR}(S(K_n,2))$. In this paper, we will show that the bounds in Theorem~\ref{thm_R_leq} are tight and determine the precise values of $\gamma_{dR}(S(K_n,t))$ for each positive integer $t$.

\medskip

This paper is organized as follows. Basic definitions and previous results are mentioned in Section~\ref{sec_intro}. In Section~\ref{sec_Dset}, a dominating set $D_{n,t}$ of the Sierpi\'{n}ski graph $S(K_n,t)$ are constructed, whose cardinality is $|D_{n,t}|=\gamma(S(K_n,t)).$ With the aid of $D_{n,t}$, in Section~\ref{sec_domNum} the domination number $\gamma(S(K_n,t)),$ Roman domination number $\gamma_R(S(K_n,t)),$ and double Roman domination $\gamma_{dR}(S(K_n,t))$ of $S(K_n,t)$ will be attained. The results are reviewed as a concluding remark in Section~\ref{sec_conclu}.

\bigskip

\section{Dominating sets $D_{n,t}$ of $S(K_n,t)$}       \label{sec_Dset}

We propose a subset $D_{n,t}$ of vertices in $S(K_n,t)$ for every pair of positive integers $n\geq 2$ and $t$ in this section. It will be shown that $D_{n,t}$ is a dominating set for $S(K_n,t).$
\begin{definition}  \label{defn_dnt}
For positive integers $n$ and $t$, let $D_{n,t}$ be a subset of the vertex set $V_{n,t}$ of $S(K_n,t)$ such that $D_{n,1}=\{1\}$ and $D_{n,2}=\{11,21,\ldots,n1\}.$ When $t\geq 3,$ for each $\textbf{v}=v_1v_2\ldots v_{t-2}\in D_{n,t-2},$ let
\begin{eqnarray}
E_1(\textbf{v}) &=& \big\{v_1v_2\ldots v_{t-3} v_{t-2} \alpha\alpha\mid \alpha\in[n]\big\},
\nonumber \\
E_2(\textbf{v}) &=& \big\{v_1v_2\ldots v_{t-3}\alpha\beta v_{t-2}\mid \alpha,\beta\in[n]\setminus\{v_{t-2}\}\big\},
\nonumber
\end{eqnarray}
and, if the entries of $\textbf{v}$ are not constant,
let $\ell=\ell(\textbf{v})$ denote the largest number in $[t-3]$ satisfying $v_\ell\neq v_{\ell+1}$ and
\begin{equation}
E_3(\textbf{v})
=\big\{v_1v_2\ldots v_{\ell-1}v_{\ell+1}v_\ell^{t-\ell-2}\alpha v_\ell\mid \alpha\in[n]\setminus\{v_\ell\}\big\}.
\nonumber
\end{equation}

Then, we define $D_{n,t}$ as follows.

\begin{enumerate}
\item[(i)] If $t\geq 3$ is odd, then
$$D_{n,t}=E_1(1^{t-2})\cup E_2(1^{t-2})\cup\bigcup\limits_{\textbf{v}\in D_{n,t-2}\setminus\{1^{t-2}\}}E_1(\textbf{v})\cup E_2(\textbf{v})\cup E_3(\textbf{v}).$$

\item[(ii)] If $t\geq 4$ is even, then
$$D_{n,t}=\{1^{t-2}\alpha 1\mid \alpha\in [n]\}\cup\bigcup\limits_{\textbf{v}\in D_{n,t-2}\setminus\{1^{t-2}\}}E_1(\textbf{v})\cup E_2(\textbf{v})\cup E_3(\textbf{v}).$$

\end{enumerate}

\end{definition}

\medskip

The sets $D_{n,t}$ are constructed inductively by odd and even $t$, respectively. For example, one can see the subset $D_{3,4}$ of vertices in $S(K_3,4)$ in Figure 2.

\medskip

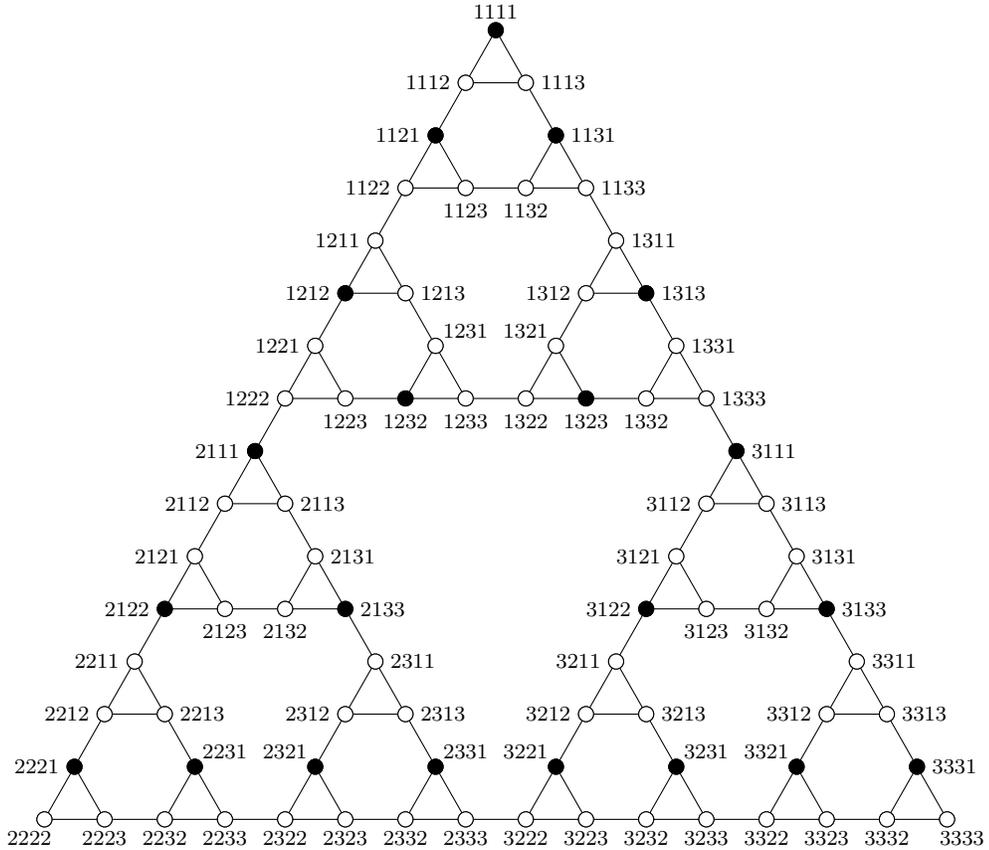
\begin{center}
\begin{tikzpicture}
\draw[] (0,10.5)coordinate(v1111) circle(0.1cm);
\draw[] (0,10.75)coordinate(1111) node{\footnotesize 1111};
\draw[] (-6,0)coordinate(v2222) circle(0.1cm);
\draw[] (-6.2,-0.25)coordinate(2222) node{\footnotesize 2222};
\draw[] (6,0)coordinate(v3333) circle(0.1cm);
\draw[] (6.2,-0.25)coordinate(3333) node{\footnotesize 3333};
\draw[] (-5.2,0)coordinate(v2223) circle(0.1cm);
\draw[] (-5.2,-0.25)coordinate(2223) node{\footnotesize 2223};
\draw[] (-4.4,0)coordinate(v2232) circle(0.1cm);
\draw[] (-4.4,-0.25)coordinate(2232) node{\footnotesize 2232};
\draw[] (-3.6,0)coordinate(v2233) circle(0.1cm);
\draw[] (-3.6,-0.25)coordinate(2233) node{\footnotesize 2233};
\draw[] (-2.8,0)coordinate(v2322) circle(0.1cm);
\draw[] (-2.8,-0.25)coordinate(2322) node{\footnotesize 2322};
\draw[] (-2,0)coordinate(v2323) circle(0.1cm);
\draw[] (-2,-0.25)coordinate(2323) node{\footnotesize 2323};
\draw[] (-1.2,0)coordinate(v2332) circle(0.1cm);
\draw[] (-1.2,-0.25)coordinate(2332) node{\footnotesize 2332};
\draw[] (-0.4,0)coordinate(v2333) circle(0.1cm);
\draw[] (-0.4,-0.25)coordinate(2333) node{\footnotesize 2333};
\draw[] (5.2,0)coordinate(v3332) circle(0.1cm);
\draw[] (5.2,-0.25)coordinate(3332) node{\footnotesize 3332};
\draw[] (4.4,0)coordinate(v3323) circle(0.1cm);
\draw[] (4.4,-0.25)coordinate(3323) node{\footnotesize 3323};
\draw[] (3.6,0)coordinate(v3322) circle(0.1cm);
\draw[] (3.6,-0.25)coordinate(3322) node{\footnotesize 3322};
\draw[] (2.8,0)coordinate(v3233) circle(0.1cm);
\draw[] (2.8,-0.25)coordinate(3233) node{\footnotesize 3233};
\draw[] (2,0)coordinate(v3232) circle(0.1cm);
\draw[] (2,-0.25)coordinate(3232) node{\footnotesize 3232};
\draw[] (1.2,0)coordinate(v3223) circle(0.1cm);
\draw[] (1.2,-0.25)coordinate(3223) node{\footnotesize 3223};
\draw[] (0.4,0)coordinate(v3222) circle(0.1cm);
\draw[] (0.4,-0.25)coordinate(3222) node{\footnotesize 3222};
\draw[] (-0.4,9.8)coordinate(v1112) circle(0.1cm);
\draw[] (-0.9,9.8)coordinate(1112) node{\footnotesize 1112};
\draw[] (-0.8,9.1)coordinate(v1121) circle(0.1cm);
\draw[] (-1.3,9.1)coordinate(1121) node{\footnotesize 1121};
\draw[] (-1.2,8.4)coordinate(v1122) circle(0.1cm);
\draw[] (-1.7,8.4)coordinate(1122) node{\footnotesize 1122};
\draw[] (-1.6,7.7)coordinate(v1211) circle(0.1cm);
\draw[] (-2.1,7.7)coordinate(1211) node{\footnotesize 1211};
\draw[] (-2,7)coordinate(v1212) circle(0.1cm);
\draw[] (-2.5,7)coordinate(1212) node{\footnotesize 1212};
\draw[] (-2.4,6.3)coordinate(v1221) circle(0.1cm);
\draw[] (-2.9,6.3)coordinate(1221) node{\footnotesize 1221};
\draw[] (-2.8,5.6)coordinate(v1222) circle(0.1cm);
\draw[] (-3.3,5.6)coordinate(1222) node{\footnotesize 1222};
\draw[] (-3.2,4.9)coordinate(v2111) circle(0.1cm);
\draw[] (-3.7,4.9)coordinate(2111) node{\footnotesize 2111};
\draw[] (-3.6,4.2)coordinate(v2112) circle(0.1cm);
\draw[] (-4.1,4.2)coordinate(2112) node{\footnotesize 2112};
\draw[] (-4,3.5)coordinate(v2121) circle(0.1cm);
\draw[] (-4.5,3.5)coordinate(2121) node{\footnotesize 2121};
\draw[] (-4.4,2.8)coordinate(v2122) circle(0.1cm);
\draw[] (-4.9,2.8)coordinate(2122) node{\footnotesize 2122};
\draw[] (-4.8,2.1)coordinate(v2211) circle(0.1cm);
\draw[] (-5.3,2.1)coordinate(2211) node{\footnotesize 2211};
\draw[] (-5.2,1.4)coordinate(v2212) circle(0.1cm);
\draw[] (-5.7,1.4)coordinate(2212) node{\footnotesize 2212};
\draw[] (-5.6,0.7)coordinate(v2221) circle(0.1cm);
\draw[] (-6.1,0.7)coordinate(2221) node{\footnotesize 2221};
\draw[] (0.4,9.8)coordinate(v1113) circle(0.1cm);
\draw[] (0.9,9.8)coordinate(1113) node{\footnotesize 1113};
\draw[] (0.8,9.1)coordinate(v1131) circle(0.1cm);
\draw[] (1.3,9.1)coordinate(1131) node{\footnotesize 1131};
\draw[] (1.2,8.4)coordinate(v1133) circle(0.1cm);
\draw[] (1.7,8.4)coordinate(1133) node{\footnotesize 1133};
\draw[] (1.6,7.7)coordinate(v1311) circle(0.1cm);
\draw[] (2.1,7.7)coordinate(1311) node{\footnotesize 1311};
\draw[] (2,7)coordinate(v1313) circle(0.1cm);
\draw[] (2.5,7)coordinate(1313) node{\footnotesize 1313};
\draw[] (2.4,6.3)coordinate(v1331) circle(0.1cm);
\draw[] (2.9,6.3)coordinate(1331) node{\footnotesize 1331};
\draw[] (2.8,5.6)coordinate(v1333) circle(0.1cm);
\draw[] (3.3,5.6)coordinate(1333) node{\footnotesize 1333};
\draw[] (3.2,4.9)coordinate(v3111) circle(0.1cm);
\draw[] (3.7,4.9)coordinate(3111) node{\footnotesize 3111};
\draw[] (3.6,4.2)coordinate(v3113) circle(0.1cm);
\draw[] (4.1,4.2)coordinate(3113) node{\footnotesize 3113};
\draw[] (4,3.5)coordinate(v3131) circle(0.1cm);
\draw[] (4.5,3.5)coordinate(3131) node{\footnotesize 3131};
\draw[] (4.4,2.8)coordinate(v3133) circle(0.1cm);
\draw[] (4.9,2.8)coordinate(3133) node{\footnotesize 3133};
\draw[] (4.8,2.1)coordinate(v3311) circle(0.1cm);
\draw[] (5.3,2.1)coordinate(3311) node{\footnotesize 3311};
\draw[] (5.2,1.4)coordinate(v3313) circle(0.1cm);
\draw[] (5.7,1.4)coordinate(3313) node{\footnotesize 3313};
\draw[] (5.6,0.7)coordinate(v3331) circle(0.1cm);
\draw[] (6.1,0.7)coordinate(3331) node{\footnotesize 3331};
\draw[] (-0.4,8.4)coordinate(v1123) circle(0.1cm);
\draw[] (-0.4,8.1)coordinate(1123) node{\footnotesize 1123};
\draw[] (0.4,8.4)coordinate(v1132) circle(0.1cm);
\draw[] (0.4,8.1)coordinate(1132) node{\footnotesize 1132};
\draw[] (-1.2,7)coordinate(v1213) circle(0.1cm);
\draw[] (-0.7,7)coordinate(1213) node{\footnotesize 1213};
\draw[] (1.2,7)coordinate(v1312) circle(0.1cm);
\draw[] (0.7,7)coordinate(1312) node{\footnotesize 1312};
\draw[] (-0.8,6.3)coordinate(v1231) circle(0.1cm);
\draw[] (-0.4,6.5)coordinate(1231) node{\footnotesize 1231};
\draw[] (0.8,6.3)coordinate(v1321) circle(0.1cm);
\draw[] (0.4,6.5)coordinate(1321) node{\footnotesize 1321};
\draw[] (-2,5.6)coordinate(v1223) circle(0.1cm);
\draw[] (-2,5.3)coordinate(1223) node{\footnotesize 1223};
\draw[] (-1.2,5.6)coordinate(v1232) circle(0.1cm);
\draw[] (-1.2,5.3)coordinate(1232) node{\footnotesize 1232};
\draw[] (-0.4,5.6)coordinate(v1233) circle(0.1cm);
\draw[] (-0.4,5.3)coordinate(1233) node{\footnotesize 1233};
\draw[] (2,5.6)coordinate(v1332) circle(0.1cm);
\draw[] (2,5.3)coordinate(1332) node{\footnotesize 1332};
\draw[] (1.2,5.6)coordinate(v1323) circle(0.1cm);
\draw[] (1.2,5.3)coordinate(1323) node{\footnotesize 1323};
\draw[] (0.4,5.6)coordinate(v1322) circle(0.1cm);
\draw[] (0.4,5.3)coordinate(1322) node{\footnotesize 1322};
\draw[] (-2.8,4.2)coordinate(v2113) circle(0.1cm);
\draw[] (-2.3,4.2)coordinate(2113) node{\footnotesize 2113};
\draw[] (2.8,4.2)coordinate(v3112) circle(0.1cm);
\draw[] (2.3,4.2)coordinate(3112) node{\footnotesize 3112};
\draw[] (-2.4,3.5)coordinate(v2131) circle(0.1cm);
\draw[] (-1.9,3.5)coordinate(2131) node{\footnotesize 2131};
\draw[] (2.4,3.5)coordinate(v3121) circle(0.1cm);
\draw[] (1.9,3.5)coordinate(3121) node{\footnotesize 3121};
\draw[] (-3.6,2.8)coordinate(v2123) circle(0.1cm);
\draw[] (-3.6,2.5)coordinate(2123) node{\footnotesize 2123};
\draw[] (-2.8,2.8)coordinate(v2132) circle(0.1cm);
\draw[] (-2.8,2.5)coordinate(2132) node{\footnotesize 2132};
\draw[] (-2,2.8)coordinate(v2133) circle(0.1cm);
\draw[] (-1.5,2.8)coordinate(2133) node{\footnotesize 2133};
\draw[] (3.6,2.8)coordinate(v3132) circle(0.1cm);
\draw[] (3.6,2.5)coordinate(3132) node{\footnotesize 3132};
\draw[] (2.8,2.8)coordinate(v3123) circle(0.1cm);
\draw[] (2.8,2.5)coordinate(3123) node{\footnotesize 3123};
\draw[] (2,2.8)coordinate(v3122) circle(0.1cm);
\draw[] (1.5,2.8)coordinate(3122) node{\footnotesize 3122};
\draw[] (-1.6,2.1)coordinate(v2311) circle(0.1cm);
\draw[] (-1.1,2.1)coordinate(2311) node{\footnotesize 2311};
\draw[] (1.6,2.1)coordinate(v3211) circle(0.1cm);
\draw[] (1.1,2.1)coordinate(3211) node{\footnotesize 3211};
\draw[] (-4.4,1.4)coordinate(v2213) circle(0.1cm);
\draw[] (-3.9,1.4)coordinate(2213) node{\footnotesize 2213};
\draw[] (-2,1.4)coordinate(v2312) circle(0.1cm);
\draw[] (-2.5,1.4)coordinate(2312) node{\footnotesize 2312};
\draw[] (-1.2,1.4)coordinate(v2313) circle(0.1cm);
\draw[] (-0.7,1.4)coordinate(2313) node{\footnotesize 2313};
\draw[] (4.4,1.4)coordinate(v3312) circle(0.1cm);
\draw[] (3.9,1.4)coordinate(3312) node{\footnotesize 3312};
\draw[] (2,1.4)coordinate(v3213) circle(0.1cm);
\draw[] (2.5,1.4)coordinate(3213) node{\footnotesize 3213};
\draw[] (1.2,1.4)coordinate(v3212) circle(0.1cm);
\draw[] (0.7,1.4)coordinate(3212) node{\footnotesize 3212};
\draw[] (-4,0.7)coordinate(v2231) circle(0.1cm);
\draw[] (-3.6,0.9)coordinate(2231) node{\footnotesize 2231};
\draw[] (-2.4,0.7)coordinate(v2321) circle(0.1cm);
\draw[] (-2.8,0.9)coordinate(2321) node{\footnotesize 2321};
\draw[] (-0.8,0.7)coordinate(v2331) circle(0.1cm);
\draw[] (-0.4,0.9)coordinate(2331) node{\footnotesize 2331};
\draw[] (4,0.7)coordinate(v3321) circle(0.1cm);
\draw[] (3.6,0.9)coordinate(3321) node{\footnotesize 3321};
\draw[] (2.4,0.7)coordinate(v3231) circle(0.1cm);
\draw[] (2.8,0.9)coordinate(3231) node{\footnotesize 3231};
\draw[] (0.8,0.7)coordinate(v3221) circle(0.1cm);
\draw[] (0.4,0.9)coordinate(3221) node{\footnotesize 3221};
\draw(v1111)--(v2222)--(v3333)--(v1111);
\draw(v1112)--(v1113);
\draw(v1121)--(v1123);
\draw(v1131)--(v1132);
\draw(v1122)--(v1133);
\draw(v1211)--(v1233);
\draw(v1311)--(v1322);
\draw(v1222)--(v1333);
\draw(v1212)--(v1213);
\draw(v1312)--(v1313);
\draw(v1221)--(v1223);
\draw(v1231)--(v1232);
\draw(v1331)--(v1332);
\draw(v1321)--(v1323);
\draw(v2111)--(v2333);
\draw(v3111)--(v3222);
\draw(v2112)--(v2113);
\draw(v3113)--(v3112);
\draw(v2122)--(v2133);
\draw(v3122)--(v3133);
\draw(v2121)--(v2123);
\draw(v2131)--(v2132);
\draw(v3131)--(v3132);
\draw(v3121)--(v3123);
\draw(v2211)--(v2233);
\draw(v2311)--(v2322);
\draw(v3311)--(v3322);
\draw(v3211)--(v3233);
\draw(v2212)--(v2213);
\draw(v2312)--(v2313);
\draw(v3313)--(v3312);
\draw(v3213)--(v3212);
\draw(v2221)--(v2223);
\draw(v2231)--(v2232);
\draw(v2321)--(v2323);
\draw(v2331)--(v2332);
\draw(v3331)--(v3332);
\draw(v3321)--(v3323);
\draw(v3231)--(v3232);
\draw(v3221)--(v3223);

\draw[fill] (0,10.5)coordinate(v1111) circle(0.1cm);
\draw[fill=white] (-6,0)coordinate(v2222) circle(0.1cm);
\draw[fill=white] (6,0)coordinate(v3333) circle(0.1cm);
\draw[fill=white] (-5.2,0)coordinate(v2223) circle(0.1cm);
\draw[fill=white] (-4.4,0)coordinate(v2232) circle(0.1cm);
\draw[fill=white] (-3.6,0)coordinate(v2233) circle(0.1cm);
\draw[fill=white] (-2.8,0)coordinate(v2322) circle(0.1cm);
\draw[fill=white] (-2,0)coordinate(v2323) circle(0.1cm);
\draw[fill=white] (-1.2,0)coordinate(v2332) circle(0.1cm);
\draw[fill=white] (-0.4,0)coordinate(v2333) circle(0.1cm);
\draw[fill=white] (5.2,0)coordinate(v3332) circle(0.1cm);
\draw[fill=white] (4.4,0)coordinate(v3323) circle(0.1cm);
\draw[fill=white] (3.6,0)coordinate(v3322) circle(0.1cm);
\draw[fill=white] (2.8,0)coordinate(v3233) circle(0.1cm);
\draw[fill=white] (2,0)coordinate(v3232) circle(0.1cm);
\draw[fill=white] (1.2,0)coordinate(v3223) circle(0.1cm);
\draw[fill=white] (0.4,0)coordinate(v3222) circle(0.1cm);
\draw[fill=white] (-0.4,9.8)coordinate(v1112) circle(0.1cm);
\draw[fill] (-0.8,9.1)coordinate(v1121) circle(0.1cm);
\draw[fill=white] (-1.2,8.4)coordinate(v1122) circle(0.1cm);
\draw[fill=white] (-1.6,7.7)coordinate(v1211) circle(0.1cm);
\draw[fill] (-2,7)coordinate(v1212) circle(0.1cm);
\draw[fill=white] (-2.4,6.3)coordinate(v1221) circle(0.1cm);
\draw[fill=white] (-2.8,5.6)coordinate(v1222) circle(0.1cm);
\draw[fill] (-3.2,4.9)coordinate(v2111) circle(0.1cm);
\draw[fill=white] (-3.6,4.2)coordinate(v2112) circle(0.1cm);
\draw[fill=white] (-4,3.5)coordinate(v2121) circle(0.1cm);
\draw[fill] (-4.4,2.8)coordinate(v2122) circle(0.1cm);
\draw[fill=white] (-4.8,2.1)coordinate(v2211) circle(0.1cm);
\draw[fill=white] (-5.2,1.4)coordinate(v2212) circle(0.1cm);
\draw[fill] (-5.6,0.7)coordinate(v2221) circle(0.1cm);
\draw[fill=white] (0.4,9.8)coordinate(v1113) circle(0.1cm);
\draw[fill] (0.8,9.1)coordinate(v1131) circle(0.1cm);
\draw[fill=white] (1.2,8.4)coordinate(v1133) circle(0.1cm);
\draw[fill=white] (1.6,7.7)coordinate(v1311) circle(0.1cm);
\draw[fill] (2,7)coordinate(v1313) circle(0.1cm);
\draw[fill=white] (2.4,6.3)coordinate(v1331) circle(0.1cm);
\draw[fill=white] (2.8,5.6)coordinate(v1333) circle(0.1cm);
\draw[fill] (3.2,4.9)coordinate(v3111) circle(0.1cm);
\draw[fill=white] (3.6,4.2)coordinate(v3113) circle(0.1cm);
\draw[fill=white] (4,3.5)coordinate(v3131) circle(0.1cm);
\draw[fill] (4.4,2.8)coordinate(v3133) circle(0.1cm);
\draw[fill=white] (4.8,2.1)coordinate(v3311) circle(0.1cm);
\draw[fill=white] (5.2,1.4)coordinate(v3313) circle(0.1cm);
\draw[fill] (5.6,0.7)coordinate(v3331) circle(0.1cm);
\draw[fill=white] (-0.4,8.4)coordinate(v1123) circle(0.1cm);
\draw[fill=white] (0.4,8.4)coordinate(v1132) circle(0.1cm);
\draw[fill=white] (-1.2,7)coordinate(v1213) circle(0.1cm);
\draw[fill=white] (1.2,7)coordinate(v1312) circle(0.1cm);
\draw[fill=white] (-0.8,6.3)coordinate(v1231) circle(0.1cm);
\draw[fill=white] (0.8,6.3)coordinate(v1321) circle(0.1cm);
\draw[fill=white] (-2,5.6)coordinate(v1223) circle(0.1cm);
\draw[fill] (-1.2,5.6)coordinate(v1232) circle(0.1cm);
\draw[fill=white] (-0.4,5.6)coordinate(v1233) circle(0.1cm);
\draw[fill=white] (2,5.6)coordinate(v1332) circle(0.1cm);
\draw[fill] (1.2,5.6)coordinate(v1323) circle(0.1cm);
\draw[fill=white] (0.4,5.6)coordinate(v1322) circle(0.1cm);
\draw[fill=white] (-2.8,4.2)coordinate(v2113) circle(0.1cm);
\draw[fill=white] (2.8,4.2)coordinate(v3112) circle(0.1cm);
\draw[fill=white] (-2.4,3.5)coordinate(v2131) circle(0.1cm);
\draw[fill=white] (2.4,3.5)coordinate(v3121) circle(0.1cm);
\draw[fill=white] (-3.6,2.8)coordinate(v2123) circle(0.1cm);
\draw[fill=white] (-2.8,2.8)coordinate(v2132) circle(0.1cm);
\draw[fill] (-2,2.8)coordinate(v2133) circle(0.1cm);
\draw[fill=white] (3.6,2.8)coordinate(v3132) circle(0.1cm);
\draw[fill=white] (2.8,2.8)coordinate(v3123) circle(0.1cm);
\draw[fill] (2,2.8)coordinate(v3122) circle(0.1cm);
\draw[fill=white] (-1.6,2.1)coordinate(v2311) circle(0.1cm);
\draw[fill=white] (1.6,2.1)coordinate(v3211) circle(0.1cm);
\draw[fill=white] (-4.4,1.4)coordinate(v2213) circle(0.1cm);
\draw[fill=white] (-2,1.4)coordinate(v2312) circle(0.1cm);
\draw[fill=white] (-1.2,1.4)coordinate(v2313) circle(0.1cm);
\draw[fill=white] (4.4,1.4)coordinate(v3312) circle(0.1cm);
\draw[fill=white] (2,1.4)coordinate(v3213) circle(0.1cm);
\draw[fill=white] (1.2,1.4)coordinate(v3212) circle(0.1cm);
\draw[fill] (-4,0.7)coordinate(v2231) circle(0.1cm);
\draw[fill] (-2.4,0.7)coordinate(v2321) circle(0.1cm);
\draw[fill] (-0.8,0.7)coordinate(v2331) circle(0.1cm);
\draw[fill] (4,0.7)coordinate(v3321) circle(0.1cm);
\draw[fill] (2.4,0.7)coordinate(v3231) circle(0.1cm);
\draw[fill] (0.8,0.7)coordinate(v3221) circle(0.1cm);

\end{tikzpicture}

\bigskip

{\bf Figure 2:} The set $D_{3,4}$ of filled vertices in $S(K_3,4).$
\end{center}

\medskip

In the rest of this section, we aim to describe the properties of $D_{n,t}.$
\begin{remark}  \label{rem_dnt}
Some quick observations involved with $D_{n,t}$ are given below.
\begin{enumerate}
\item[(i)] We verify that $1^t \in D_{n,t}$ while all of the vertices $\{\alpha^t\mid \alpha\in [n]\setminus\{1\}\}$ are not in $D_{n,t}$ so that Definition~\ref{defn_dnt} is well-defined. Since the vertices in $E_2(\textbf{v})$ and $E_3(\textbf{v})$ are obviously with non-constant entries, we focus on $E_1(\textbf{v}).$ Clearly, $E_1(\textbf{v})$ contains a vertex in $D_{n,t}$ of constant entry if and only if $\textbf{v}$ is a vertex in $D_{n,t-2}$ of constant entry. Moreover, $1$ and $11$ are the only vertices in $D_{n,1}$ and $D_{n,2}$ of constant entry, respectively. The claim immediately follows by induction.
\item[(ii)] Directly from Definition~\ref{defn_dnt}, $|E_1(\textbf{v})|=n$ and $|E_2(\textbf{v})|=(n-1)^2$ for $\textbf{v}\in D_{n,t-2}$ (if applicable), and $|E_3(\textbf{v})|=n-1$ for $\textbf{v}\in D_{n,t-2}\setminus\{1^{t-2}\}.$
\end{enumerate}
\end{remark}

\medskip

To simplify the notation, let $D_{n,t}^*$ denote the set $D_{n,t}\setminus\{1^t\}$ throughout this paper.
\begin{lemma}   \label{lem_disj_union}
Let $n,t$ be positive integers not less than 3. If $t$ is odd, then the sets $E_1(\textbf{u}),$ $E_2(\textbf{v}),$ and $E_3(\textbf{w})$ are pairwise disjoint for $\textbf{u},\textbf{v}\in D_{n,t-2},$ and $\textbf{w}\in D_{n,t-2}^*.$ If $t$ is even, then the sets $E_1(\textbf{u}),$ $E_2(\textbf{v}),$ $E_3(\textbf{w}),$ and $\{1^{t-2}\alpha1\mid \alpha\in[n]\}$ are pairwise disjoint for $\textbf{u},\textbf{v},\textbf{w}\in D_{n,t-2}^*.$ Additionally, for each $1\leq i\leq 3,$ $E_i(\textbf{u})$ and $E_i(\textbf{v})$ are disjoint if $\textbf{u}\neq \textbf{v}$ in their proper domain.
\end{lemma}
\begin{proof}
By Definition~\ref{defn_dnt}, every vertex in $E_1(\textbf{u})$ has the same entries in the last two entries, while different for each vertex in $E_2(\textbf{v})\cup E_3(\textbf{w}).$ Therefore, $E_1(\textbf{u})$ and $E_2(\textbf{v})\cup E_3(\textbf{w})$ have no intersection. Also, the $(t-2)$-th and last entries are identical for every vertex in $E_3(\textbf{w})$, while distinct for any vertex in $E_2(\textbf{v}).$ Hence, $E_2(\textbf{v})$ and $E_3(\textbf{w})$ have no intersection. Then, we deal with the set $\{1^{t-2}\alpha1\mid \alpha\in[n]\}$ as $n$ is even. For $1\leq i\leq 3$ and $\textbf{u}\in D_{n,t-2}^*,$ we notice that all vertices in $E_i(\textbf{u})$ do not have ones on all entries other than the $(t-1)$-th position, and thus $E_i(\textbf{u})\cap\{1^{t-2}\alpha 1\mid\alpha\in[n]\}$ is empty. Lastly, for $1\leq i\leq 3$ and $\textbf{u}\neq \textbf{v}$ in their proper domain, the fact $E_i(\textbf{u})\cap E_i(\textbf{v})=\emptyset$ can be attained directly from the definition of $E_i$. The result follows.
\end{proof}

\medskip

Immediately from Lemma~\ref{lem_disj_union}, we count the number of vertices in $D_{n,t}.$
\begin{lemma}   \label{lem_cardDnt}
For positive integers $n\geq 2$ and $t,$ the cardinality of $D_{n,t}$ is
$$|D_{n,t}|=\left\lceil\frac{n^t}{n+1}\right\rceil.$$
\end{lemma}
\begin{proof}
We fix $n$ and prove the result by induction on $t$ in 2 cases: $t$ is odd and $t$ is even.

\smallskip

When $t$ is odd, we have $|D_{n,1}|=|\{1\}|=1=\left\lceil n/(n+1)\right\rceil.$ Suppose that $|D_{n,t-2}|=\lceil n^{t-2}/(n+1)\rceil$ for some odd $t\geq 3.$
Then by Definition~\ref{defn_dnt}(i),
\begin{eqnarray}
|D_{n,t}| &=& n+(n-1)^2+n^2\cdot(|D_{n,t-2}|-1)
\label{eq_cardDnt_odd} \\
&=& n+(n-1)^2+n^2\cdot\left(\frac{n^{t-2}+1}{n+1}-1\right)
\nonumber \\
&=& \frac{n^t+1}{n+1} = \left\lceil \frac{n^t}{n+1}\right\rceil,
\nonumber
\end{eqnarray}
where the equality in \eqref{eq_cardDnt_odd} can be referred to Remark~\ref{rem_dnt}(ii).

\smallskip

Let $t$ be even. Then $|D_{n,2}|=|\{\alpha1\mid\alpha\in[n]\}|=n=\left\lceil n^2/(n+1)\right\rceil.$ Assume that $|D_{n,t-2}|=\left\lceil n^{t-2}/(n+1)\right\rceil$ for some even $t\geq 4.$ Then by Definition~\ref{defn_dnt}(ii),
\begin{eqnarray}
|D_{n,t}| &=& n+n^2\cdot(|D_{n,t-2}|-1)
\label{eq_cardDnt_even} \\
&=& n+n^2\cdot\left(\frac{n^{t-2}+n}{n+1}-1\right)
\nonumber \\
&=& \frac{n^t+n}{n+1} = \left\lceil \frac{n^t}{n+1}\right\rceil,
\nonumber
\end{eqnarray}
where the equality in \eqref{eq_cardDnt_even} is referred to Remark~\ref{rem_dnt}(ii). The result follows.
\end{proof}

\medskip

In Lemma~\ref{lem_disj_union}, we show that all vertices are distinct in $E_1,$ $E_2,$ and $E_3.$ Moreover, by observing the neighborhood, they are separated far away. Recall that the {\it distance} between two vertices in a simple graph is the number of edges in a shortest path connecting them.
\begin{lemma}   \label{lem_noCommonN}
Let $n\geq 2$ and $t$ be positive integers.
\begin{enumerate}
\item[(i)] If $t$ is odd, then every pair of distinct vertices in $D_{n,t}$ have distance at least $3$ in $S(K_n,t).$
\item[(ii)] If $t$ is even, then every pair of distinct vertices in $D_{n,t}^*$ have distance at least $3$ in $S(K_n,t).$
\end{enumerate}
\end{lemma}
\begin{proof} For a vertex $\textbf{v}$ in $S(K_n,t),$ let $N_{n,t}[\textbf{v}]$ denote the set containing the vertex $\textbf{v}$ and its neighbors in $S(K_n,t)$. Equivalently, two vertices $\textbf{u}$ and $\textbf{v}$ have distance at least $3$ if and only if $N_{n,t}[\textbf{u}]$ and $N_{n,t}[\textbf{v}]$ have no intersection. In the following, we prove the result by induction on $t$ as $t$ is odd and even, respectively.

\smallskip

For (i), when $t$ is odd, the result holds for $D_{n,1}=\{1\},$ and we assume that it holds for $D_{n,t-2}$ for some odd $t\geq 3.$ By observing Definition~\ref{defn_dnt}, for each $1\leq i\leq 3$ if $\textbf{x}\in E_i(\textbf{u}),$ then the first $t-2$ entries of a vertex in $N_{n,t}[\textbf{x}]$ must be a vertex in $N_{n,t-2}[\textbf{u}].$ Therefore, for two distinct vertices $\textbf{x}\in E_i(\textbf{u})$ and $\textbf{y}\in E_j(\textbf{v})$ in $D_{n,t},$ where $1\leq i,j\leq 3,$ if $\textbf{u}$ and $\textbf{v}$ are distinct vertices in $D_{n,t-2}$ then $N_{n,t}[\textbf{x}]$ and $N_{n,t}[\textbf{y}]$ must have no intersection, or otherwise the first $t-2$ entries of an element in $N_{n,t}[\textbf{x}]\cap N_{n,t}[\textbf{y}]$ will be a vertex in $N_{n,t-2}[\textbf{u}]\cap N_{n,t-2}[\textbf{v}]$ so that the distance of $\textbf{u}$ and $\textbf{v}$ is less than $3,$ which contradicts to the induction hypothesis. Now, suppose that $\textbf{x}$ and $\textbf{y}$ are two distinct vertices in $D_{n,t}$ such that $\textbf{x}\in E_i(\textbf{u})$ and $\textbf{y}\in E_j(\textbf{u}),$ for some $1\leq i,j\leq 3.$ Additionally, for a non-constant vertex $\textbf{v}=v_1\ldots v_t$ in $S(K_n,t),$ let $\textbf{v}^\vdash$ denote the unique neighbor $v_1\ldots v_{\ell-1}v_{\ell+1}v_{\ell}\ldots v_{\ell}$ of $\textbf{v}$ obtained by flipping the entries, where $1\leq\ell<t$ is the largest integer such that $v_\ell\neq v_{\ell+1}.$ Note that $(\textbf{v}^\vdash)^\vdash=\textbf{v},$ and hence $\textbf{x}=\textbf{y}$ if and only if $\textbf{x}^\vdash=\textbf{y}^\vdash.$ The discussion can be partitioned into the following cases.

\smallskip

\noindent{\textbf{Case 1. $i=1$ and $j=1.$}}\\
Assume that $\textbf{x}=\{u_1\ldots u_{t-2}\alpha\alpha\}$ and $\textbf{y}=\{u_1\ldots u_{t-2}\beta\beta\}$, where $\alpha\neq\beta.$ Then $N_{n,t}[\textbf{x}]=\{u_1\ldots,u_{t-2}\alpha\alpha'\mid\alpha'\in[n]\}\cup\{\textbf{x}^\vdash\},$ where $\textbf{x}^\vdash$ exists if $\textbf{x}$ is not the all ones vertex. Therefore, if $N_{n,t}[\textbf{x}]\cap N_{n,t}[\textbf{y}]$ is nonempty, then we may assume
\begin{equation}    \label{eq_i1j1}
\textbf{x}^\vdash=u_1\ldots u_{\ell-1}\alpha u_\ell\ldots u_\ell=u_1\ldots u_{t-2}\beta\lambda\in N_{n,t}[\textbf{y}],
\end{equation}
where $\ell\leq t-2$ is the largest number satisfying $u_\ell\neq \alpha$ and $\lambda\in[n].$ However, the $\ell$-th entry in~\eqref{eq_i1j1} implies $u_\ell=\alpha,$ which is a contradiction.

\smallskip

\noindent{\textbf{Case 2. $i=1$ and $j=2.$}}\\
Assume that $\textbf{x}=\{u_1\ldots u_{t-2}\alpha\alpha\}$ and $\textbf{y}=\{u_1\ldots u_{t-3}\beta\lambda u_{t-2}\}$, where $\beta,\lambda\in [n]\setminus\{u_{t-2}\}.$ Since every vertex in $N_{n,t}[\textbf{y}]$ is of the $(t-2)$-th entry $\beta,$ hence, if $N_{n,t}[\textbf{x}]$ and $N_{n,t}[\textbf{y}]$ has intersection then the only possibility is
\begin{equation}    \label{eq_i1j2}
\textbf{x}^\vdash=u_1\ldots u_{\ell-1}\alpha u_\ell\ldots u_\ell=u_1\ldots u_{t-3}\beta\lambda\delta\in N_{n,t}[\textbf{y}],
\end{equation}
where $\ell\leq t-2$ is the largest number satisfying $u_\ell\neq \alpha$ and $\delta\in[n].$ However, if $\ell=t-2$ then a contradiction occurs because the $(t-1)$-th entry in \eqref{eq_i1j2} tells that $u_{t-2}=\lambda;$ if $\ell<t-2$ then the $\ell$-th entry in~\eqref{eq_i1j2} implies $u_\ell=\alpha,$ which attains a contradiction also.

\smallskip

\noindent{\textbf{Case 3. $i=2$ and $j=2.$}}\\
Assume that $\textbf{x}=\{u_1\ldots u_{t-3}\alpha\beta u_{t-2}\}$ and $\textbf{y}=\{u_1\ldots u_{t-3}\lambda\delta u_{t-2}\}$, where $\alpha,\beta,\lambda,\delta\in [n]\setminus\{u_{t-2}\}$ with $(\alpha,\beta)\neq (\lambda,\delta).$ Notice that that the $(t-3)$-th entries of every vertex in $N_{n,t}[\textbf{x}]$ and $N_{n,t}[\textbf{y}]$ are $\alpha$ and $\lambda,$ respectively. Hence, if $N_{n,t}[\textbf{x}]\cap N_{n,t}[\textbf{y}]$ is not empty then $\alpha=\lambda$ so that $\beta\neq \delta,$ and we may assume
\begin{equation}    \label{eq_i2j2}
\textbf{x}^\vdash=u_1\ldots u_{t-3}\alpha u_{t-2}\beta=u_1\ldots u_{t-3}\lambda\delta\epsilon\in N_{n,t}[\textbf{y}],
\end{equation}
where $\epsilon\in [n].$ However, a contradiction happens since the $(t-1)$-th entry in~\eqref{eq_i2j2} says that $\delta=u_{t-2}.$

\smallskip

\noindent{\textbf{Case 4. $i=1$ and $j=3.$}}\\
For the following cases involved with $j=3,$ let $\textbf{u}\in D_{n,t-2}^*$ and $h<t-2$ be the largest number satisfying $u_h\neq u_{h+1}.$ Assume that $\textbf{x}=\{u_1\ldots u_{t-2}\alpha\alpha\}$ and $\textbf{y}=\{u_1\ldots u_{h-1}u_{h+1}u_h\ldots u_h\beta u_h\}$, where $\beta\in [n]\setminus \{u_h\}.$ It is clear that the $h$-th entry of each vertex in $N_{n,t}[\textbf{y}]$ is $u_{h+1}.$ Thus, if there exists an element in $N_{n,t}[\textbf{x}]\cap N_{n,t}[\textbf{y}]$ then it is
\begin{equation}    \label{eq_i1j3}
\textbf{x}^\vdash=u_1\ldots u_{\ell-1}\alpha u_\ell\ldots u_\ell=u_1\ldots u_{h-1}u_{h+1}u_h\ldots u_h\beta\lambda\in N_{n,t}[\textbf{y}],
\end{equation}
where $\ell\leq t-2$ is the largest number satisfying $u_\ell\neq \alpha$ and $\lambda\in[n].$ Consequently, the only possibility is $\ell=h$ and the last $t-h$ entries are all $u_h,$ which contradicts to $\beta\neq u_h.$

\smallskip

\noindent{\textbf{Case 5. $i=2$ and $j=3.$}}\\
Assume that $\textbf{x}=\{u_1\ldots u_{t-3}\alpha\beta u_{t-2}\}$ and $\textbf{y}=\{u_1\ldots u_{h-1}u_{h+1}u_h\ldots u_h\lambda u_h\}$, where $\alpha,\beta\in [n]\setminus \{u_{t-2}\}$ and $\lambda\in [n]\setminus\{u_h\}.$ From the fact that any vertex in $N_{n,t}[\textbf{x}]$ has $h$-th entry $u_h,$ while $u_{h+1}$ for those in $N_{n,t}[\textbf{y}],$ it is obvious that $N_{n,t}[\textbf{x}]\cap N_{n,t}[\textbf{y}]$ is the empty set.

\smallskip

\noindent{\textbf{Case 6. $i=3$ and $j=3.$}}\\
Assume that $\textbf{x}=\{u_1\ldots u_{h-1}u_{h+1}u_h\ldots u_h\alpha u_h\}$ and $\textbf{y}=\{u_1\ldots u_{h-1}u_{h+1}u_h\ldots u_h\beta u_h\}$, where $\alpha,\beta\in [n]\setminus\{u_h\}$ with $\alpha\neq\beta$. One can see that both of $\textbf{x}$ and $\textbf{y}$ have distinct last 2 entries, and they are different in the $(t-1)$-th entry only. Therefore, if $N_{n,t}[\textbf{x}]\cap N_{n,t}[\textbf{y}]$ is nonempty, then we may assume
\begin{equation}    \label{eq_i3j3}
\textbf{x}^\vdash=u_1\ldots u_{h-1}u_{h+1}u_h\ldots u_h\alpha=u_1\ldots u_{h-1}u_{h+1}u_h\ldots u_h\beta\lambda\in N_{n,t}[\textbf{y}],
\end{equation}
where $\lambda\in [n].$ However, the $(t-1)$-th entry in~\eqref{eq_i3j3} indicates $\beta=u_h,$ which is a contradiction.

\smallskip

For (ii), as $t$ is even, we can check that the result holds for $D_{n,2}^*=\{\alpha 1\mid\alpha\in [n]\setminus\{1\}\},$ and assume that it holds for $D_{n,t-2}^*$ for some even $t\geq 4.$ Similar argument can be made between $E_1,$ $E_2,$ and $E_3$ as what we did for odd $t,$ but there is another set of vertices $F=\{1^{t-2}\alpha 1\mid\alpha\in [n]\setminus\{1\}\}$ in $D_{n,t}^*.$ Nevertheless, for every element $\textbf{x}$ in $F,$ the vertices in $N_{n,t}[\textbf{x}]$ are of all ones in the first $t-2$ entries, while for any $\textbf{y}\in E_i(\textbf{u})$ where $\textbf{u}\in D_{n,t-2}^*$ and $i\in\{1,3\},$ the vertices in $N_{n,t}[\textbf{y}]$ are not constant in the first $t-2$ entries since $\textbf{u}$ is not constant by Remark~\ref{rem_dnt}(i). Moreover, we can see that $D_{n,t-2}$ and $\{1^{t-3}\beta\mid\beta\in [n]\setminus\{1\}\}$ have no intersection, or otherwise $1^{t-2}$ and $1^{t-3}\beta$ are of distance $1$ in $S(K_n,t-2),$ which violates the induction hypothesis. Therefore, for $\textbf{y}\in E_{2}(\textbf{u})$ where $\textbf{u}\in D_{n,t-2}^*,$ the first $t-2$ entries of any element from $N_{n,t}[\textbf{y}]$ are not all ones so that $N_{n,t}[\textbf{x}]\cap N_{n,t}[\textbf{y}]=\emptyset$ for all $\textbf{x}\in F.$ The proof is completed.
\end{proof}

\medskip

We bring out the main property of the vertices set $D_{n,t}.$
\begin{theorem}     \label{thm_Dnt_domin_set}
For positive integers $n\geq 2$ and $t,$ $D_{n,t}$ forms a dominating set of $S(K_n,t).$
\end{theorem}
\begin{proof}
By Remark~\ref{rem_dnt}(i), each vertex in $D_{n,t}^*$ is of degree $n,$ while $1^t$ is of degree $n-1.$ If $t$ is odd, then from Lemma~\ref{lem_noCommonN}(i), $N_{n,t}[\textbf{v}]$ are pairwise disjoint for all $\textbf{v}\in D_{n,t}.$ Therefore,
\begin{eqnarray}
\left|\bigcup_{\textbf{v}\in D_{n,t}}N_{n,t}[\textbf{v}]\right|
&=& \sum_{\textbf{v}\in D_{n,t}}N_{n,t}[\textbf{v}]
\nonumber   \\
&=& (n+1)\cdot\left(\left\lceil\frac{n^t}{n+1}\right\rceil-1\right)+n\cdot 1
\label{eq_Dnt_dom_odd}  \\
&=& (n+1)\cdot\left(\frac{n^t+1}{n+1}-1\right)+n = n^t,
\nonumber
\end{eqnarray}
where~\eqref{eq_Dnt_dom_odd} is from Lemma~\ref{lem_cardDnt}.

\smallskip

Next, assume that $t$ is even. By Lemma~\ref{lem_noCommonN}(ii), $N_{n,t}[\textbf{v}]$ are pairwise disjoint for all $\textbf{v}\in D_{n,t}^*.$ Furthermore, none of the vertices in $D_{n,t}^*$ is adjacent to $1^t$ in $S(K_n,t),$ since by Definition~\ref{defn_dnt}(ii) the first $t-1$ entries of each vertex in $D_{n,t}^*$ are not all ones. As a result,
\begin{eqnarray}
\left|\bigcup_{\textbf{v}\in D_{n,t}}N_{n,t}[\textbf{v}]\right|
&=& \left|\bigcup_{\textbf{v}\in D_{n,t}^*}N_{n,t}[\textbf{v}]\right|+\left|N_{n,t}[1^t]\setminus\bigcup_{\textbf{v}\in D_{n,t}^*}N_{n,t}[\textbf{v}]\right|
\nonumber   \\
&\geq& \left(\sum_{\textbf{v}\in D_{n,t}^*}N_{n,t}[\textbf{v}]\right)+\left|\{1^t\}\right|
nonumber   \\
&=& (n+1)\cdot\left(\left\lceil\frac{n^t}{n+1}\right\rceil-1\right)+1
\label{eq_Dnt_dom_even}  \\
&=& (n+1)\cdot\left(\frac{n^t+n}{n+1}-1\right)+1 = n^t,
\nonumber
\end{eqnarray}
where~\eqref{eq_Dnt_dom_even} is from Lemma~\ref{lem_cardDnt}.

\smallskip

The above argument indicates that $D_{n,t}$ and their neighbors include all $n^t$ vertices in $S(K_n,t),$ since $\bigcup_{\textbf{v}\in D_{n,t}}N_{n,t}[\textbf{v}]$ is a subset of the vertex set of $S(K_n,t).$ The result follows.
\end{proof}

\bigskip

\section{Domination in $S(K_n,t)$}      \label{sec_domNum}

In this section, the exact values of domination numbers $\gamma(S(K_n,t))$, Roman domination numbers $\gamma_R(S(K_n,t))$, and double Roman domination numbers $\gamma_{dR}(S(K_n,t))$ of the Sierpi\'{n}ski graphs $S(K_n,t)$ are given.

\medskip

\subsection{Domination numbers}
The vertices set $D_{n,t}$ is verified to be a dominating set for $S(K_n,t)$ in Theorem~\ref{thm_Dnt_domin_set}. Its cardinality $|D_{n,t}|$ is also obtained in Lemma~\ref{lem_cardDnt}. Therefore, we may attain the domination number $\gamma(S(K_n,t))$ as follows.
\begin{theorem}     \label{thm_dom_num}
For every positive integers $n\geq 2$ and $t,$ the domination number of the Sierpi\'{n}ski graph $S(K_n,t)$ is
$$\gamma(S(K_n,t))=\left\lceil\frac{n^t}{n+1}\right\rceil.$$
\end{theorem}
\begin{proof}
Firstly, since the maximum vertex degree in $S(K_n,t)$ is $n,$ it is straightforward to see that $$\gamma(S(K_n,t))\geq\left\lceil\frac{n^t}{n+1}\right\rceil,$$
since there are $n^t$ vertices in $S(K_n,t).$

\smallskip

Next, since the set $D_{n,t}$ given in Definition~\ref{defn_dnt} is shown to be a dominating set for $S(K_n,t)$ in Theorem~\ref{thm_Dnt_domin_set}, we have
$$\gamma(S(K_n,t))\leq|D_{n,t}|=\left\lceil\frac{n^t}{n+1}\right\rceil,$$
in which the cardinality of $D_{n,t}$ is counted in Lemma~\ref{lem_cardDnt}. The proof is completed.
\end{proof}

\bigskip

\subsection{Roman domination numbers}
In this subsection, we extend the results and proof in Theorem~\ref{thm_dom_num} and obtain the Roman domination numbers of $S(K_n,t).$ Furthermore, we confirm that the equalities hold in Theorem~\ref{thm_R_leq} for all $n,t$.
\begin{theorem}     \label{thm_RDominNum}
For every positive integers $n\geq 2$ and $t,$ the Roman domination number of the Sierpi\'{n}ski graph $S(K_n,t)$ is
$$\gamma_R(S(K_n,t))=\left\{
\renewcommand\arraystretch{1.6}
\begin{array}{ll}
2\left\lceil\frac{n^t}{n+1}\right\rceil & \text{if $t$ is odd}; \\
2\left\lceil\frac{n^t}{n+1}\right\rceil-1 & \text{if $t$ is even}.
\end{array}\right.$$
\end{theorem}
\begin{proof}
Let $f:V(S(K_n,t))\rightarrow \{0,1,2\}$ be a Roman dominating function on $S(K_n,t),$ where $V(S(K_n,t))$ is the set of vertices in $S(K_n,t).$ Suppose that $V_1$ and $V_2$ are the sets of vertices in $S(K_n,t)$ that are valued with 1 and 2 in $f,$ respectively. We have
\begin{eqnarray}
n^t = |V(S(K_n,t))| &=& \left|V_1\cup\bigcup_{\textbf{v}\in V_2} N_{n,t}[\textbf{v}]\right|
\nonumber   \\
&\leq& |V_1|+\sum_{\textbf{v}\in V_2}\left|N_{n,t}[\textbf{v}]\right|
\nonumber   \\
&\leq& |V_1|+(n+1)|V_2|,
\label{eq_RoV1V2}
\end{eqnarray}
where~\eqref{eq_RoV1V2} is from the fact that the maximal vertex degree in $S(K_n,t)$ is $n$. Then, it comes to a linear program: finding $\min\{|V_1|+2|V_2|\}$ provided that the nonnegative integers $|V_1|$ and $|V_2|$ satisfying $|V_1|+(n+1)|V_2|\geq n^t.$ By comparing the slopes, $\min\{|V_1|+2|V_2|\}$ can be attained if we make $|V_1|$ as small as possible. Therefore,
\begin{eqnarray}
\gamma_R(S(K_n,t)) &\geq& \min\{|V_1|+2|V_2|\}
\nonumber   \\
&=&\left\{
\renewcommand\arraystretch{1.6}
\begin{array}{ll}
1\cdot 0+2\cdot \frac{n^t+1}{n+1} = 2\left\lceil\frac{n^t}{n+1}\right\rceil & \text{if $t$ is odd}; \\
1\cdot 1+2\cdot \frac{n^t-1}{n+1} = 2\left\lceil\frac{n^t}{n+1}\right\rceil-1 & \text{if $t$ is even}.
\end{array}\right.
\label{eq_RoV1V2_LP}
\end{eqnarray}
where~\eqref{eq_RoV1V2_LP} is obtained by letting $(|V_1|,|V_2|)=(0,(n^t+1)/(n+1))$ if $t$ is odd, and $(|V_1|,|V_2|)=(1,(n^t-1)/(n+1))$ if $t$ is even.

\smallskip

On the other hand, although the upper bound has been shown in Theorem~\ref{thm_R_leq}, we derive a Roman dominating function from the set $D_{n,t}$ and reprove it. If $t$ is odd, let
\begin{equation}
f(\textbf{v})=
\left\{\begin{array}{ll}
2 & \text{if $\textbf{v}\in D_{n,t}$};   \\
0 & \text{if $\textbf{v}\not\in D_{n,t}$}
\end{array}
\right.
\nonumber
\end{equation}
which achieves a Roman domination since $D_{n,t}$ is a dominating set of $S(K_n,t).$ Also, $f$ sums to $\sum_{\textbf{v}}f(\textbf{v})=2|D_{n,t}|=2\lceil n^t/(n+1)\rceil.$ If $t$ is even, let
\begin{equation}
f(\textbf{v})=
\left\{\begin{array}{ll}
2 & \text{if $\textbf{v}\in D_{n,t}^*$};   \\
1 & \text{if $\textbf{v}=1^t$};     \\
0 & \text{if $\textbf{v}\not\in D_{n,t}$}
\end{array}
\right.
\nonumber
\end{equation}
which attains a Roman domination, since in the latter part of proof in Theorem~\ref{thm_Dnt_domin_set}, we mention that $1^t$ is not a neighbor of any vertex in $D_{n,t}$. In this case, we have $\sum_{\textbf{v}}f(\textbf{v})=2|D_{n,t}^*|+1=2\lceil n^t/(n+1)\rceil-1.$

\smallskip

Since the lower bound and upper bound meet in the above argument, the result follows.
\end{proof}

\bigskip

\subsection{Double Roman domination numbers}
In this subsection, we give the exact values of the double Roman domination numbers $\gamma_{dR}(S(K_n,t))$ for arbitrary $n,t$, which generalize the result~\cite[Theorem~3.2]{aa:20} stating that $\gamma_{dR}(S(K_n,t))=3n-1.$ The methods of finding double Roman domination numbers of $S(K_n,t)$ will be similar to those for Roman domination numbers.
\begin{theorem}
For every positive integers $n\geq 2$ and $t,$ the double Roman domination number of the Sierpi\'{n}ski graph $S(K_n,t)$ is
$$\gamma_{dR}(S(K_n,t))=\left\{
\renewcommand\arraystretch{1.6}
\begin{array}{ll}
3\left\lceil\frac{n^t}{n+1}\right\rceil & \text{if $t$ is odd}; \\
3\left\lceil\frac{n^t}{n+1}\right\rceil-1 & \text{if $t$ is even}.
\end{array}\right.$$
\end{theorem}
\begin{proof}
By Proposition~\ref{prop_noV1}, we narrow our discussion by letting $f:V(K_n,t)\rightarrow\{0,2,3\}$ be a double Roman dominating function, where $V(K_n,t)$ is the set of vertices in $S(K_n,t).$ Suppose that $V_2$ and $V_3$ are the set of vertices in $S(K_n,t)$ that are valued with $2$ and $3$, respectively. Since a vertex valued with $0$ can be ``guarded'' by $2$ neighbors valued with $2$ in $f$, we may assume that a vertex valued with $2$ in $f$ guards at most $1+n/2$ vertices in a graph. Therefore, the equation \eqref{eq_RoV1V2} becomes
\begin{equation}
n^t = |V(S(K_n,t))| \leq \left(1+\frac{n}{2}\right)|V_2|+(n+1)|V_3|.
\label{eq_dRoV1V2}
\end{equation}
A new linear program appears as follows: finding $\min\{2|V_2|+3|V_3|\}$ provided that nonnegative integers $|V_2|$ and $|V_3|$ satisfying~\eqref{eq_dRoV1V2}. We can see that the absolute value of slope in~\eqref{eq_dRoV1V2} is $2-2/(n+2)\geq 3/2.$ Therefore, $\min\{2|V_3|+3|V_3|\}$ can be attained if $|V_2|$ is as small as possible. We have
\begin{eqnarray}
\gamma_{dR}(S(K_n,t)) &\geq& \min\{2|V_2|+3|V_3|\}
\nonumber   \\
&=&\left\{
\renewcommand\arraystretch{1.6}
\begin{array}{ll}
2\cdot 0+3\cdot \frac{n^t+1}{n+1} = 3\left\lceil\frac{n^t}{n+1}\right\rceil & \text{if $t$ is odd}; \\
2\cdot 1+3\cdot \frac{n^t-1}{n+1} = 3\left\lceil\frac{n^t}{n+1}\right\rceil-1 & \text{if $t$ is even}
\end{array}\right.
\label{eq_dRoV1V2_LP}
\end{eqnarray}
where~\eqref{eq_dRoV1V2_LP} is obtained by letting $(|V_2|,|V_3|)=(0,(n^t+1)/(n+1))$ if $t$ is odd, and $(|V_2|,|V_3|)=(1,(n^t-1)/(n+1))$ if $t$ is even.

\smallskip

On the other hand, we verify the upper bound by giving a double Roman dominating function. If $t$ is odd let
\begin{equation}
f(\textbf{v})=
\left\{\begin{array}{ll}
3 & \text{if $\textbf{v}\in D_{n,t}$};   \\
0 & \text{if $\textbf{v}\not\in D_{n,t}$}
\end{array}
\right.
\nonumber
\end{equation}
with $\sum_{\textbf{v}}f(\textbf{v})=3|D_{n,t}|=3\lceil n^t/(n+1)\rceil,$ and if $t$ is even let
\begin{equation}
f(\textbf{v})=
\left\{\begin{array}{ll}
3 & \text{if $\textbf{v}\in D_{n,t}^*$};   \\
2 & \text{if $\textbf{v}=1^t$};     \\
0 & \text{if $\textbf{v}\not\in D_{n,t}$}
\end{array}
\right.
\nonumber
\end{equation}
with $\sum_{\textbf{v}}f(\textbf{v})=3|D_{n,t}^*|+2=3\lceil n^t/(n+1)\rceil-1.$ Similar to the proof in Theorem~\ref{thm_RDominNum}, each of the above two cases reaches a double Roman domination in $S(K_n,t).$

\smallskip

Thus, the lower bound and upper bound meet. We have the proof.
\end{proof}

\bigskip

\section{Concluding remark}     \label{sec_conclu}
In this study, based on the properties of $D_{n,t}$ defined in Definition~\ref{defn_dnt}, we obtain the precise values of domination numbers $\gamma(S(K_n,t))$, Roman domination numbers $\gamma_R(S(K_n,t))$ and double Roman domination numbers $\gamma_{dR}(S(K_n,t))$ of Sierpi\'{n}ski graphs $S(K_n,t).$ As applications, we improve Theorem~\ref{thm_R_leq} given in~\cite{rrr:17} by showing that the equality hold for any pair of $n$ and $t$. Moreover, since $\gamma_R(S(K_n,2))$ and $\gamma_{dR}(S(K_n,2))$ have been obtained in~\cite{aa:20}, our work also extend the their results to arbitrary $t.$ To conclude this paper, one can refer to the following table.

\medskip

\noindent
\begin{center}
\renewcommand\arraystretch{1.5}
\begin{tabular}{c|c|c}
  \begin{tabular}{cc}
  Sierpi\'{n}ski graphs\\
  $S(K_n,t)$
  \end{tabular}
  & Numbers &
  \begin{tabular}{c}
  Dominating sets\\
  or functions
  \end{tabular} \\
  \hline
  Domination &  $\gamma(S(K_n,t))=\left\lceil\frac{n^t}{n+1}\right\rceil $ & $D_{n,t}$ \\
  \hline
  Roman domination &
  \begin{tabular}{l}
  $\gamma_R(S(K_n,t))$\\
  $=\left\{
    \begin{array}{ll}
    2\left\lceil\frac{n^t}{n+1}\right\rceil & \text{if $t$ is odd} \\
    2\left\lceil\frac{n^t}{n+1}\right\rceil-1 & \text{if $t$ is even}
    \end{array}\right.$
  \end{tabular}
  &
  \begin{tabular}{l}
  When $t$ is odd,\\
  $f(\textbf{v})=
    \left\{\begin{array}{ll}
    2 & \text{if $\textbf{v}\in D_{n,t}$}  \\
    0 & \text{if $\textbf{v}\not\in D_{n,t}$}
    \end{array}
    \right.$\\
  When $t$ is even,\\
  $f(\textbf{v})=
    \left\{\begin{array}{ll}
    2 & \text{if $\textbf{v}\in D_{n,t}^*$}   \\
    1 & \text{if $\textbf{v}=1^t$}    \\
    0 & \text{if $\textbf{v}\not\in D_{n,t}$}
    \end{array}
    \right.$
  \end{tabular}
  \\
  \hline
  \begin{tabular}{cc}
  Double Roman\\
  domination
  \end{tabular} &
  \begin{tabular}{l}
  $\gamma_{dR}(S(K_n,t))$\\
  $=\left\{
    \begin{array}{ll}
    3\left\lceil\frac{n^t}{n+1}\right\rceil & \text{if $t$ is odd} \\
    3\left\lceil\frac{n^t}{n+1}\right\rceil-1 & \text{if $t$ is even}
    \end{array}\right.$
  \end{tabular}
  &
  \begin{tabular}{l}
  When $t$ is odd,\\
  $f(\textbf{v})=
    \left\{\begin{array}{ll}
    3 & \text{if $\textbf{v}\in D_{n,t}$}   \\
    0 & \text{if $\textbf{v}\not\in D_{n,t}$}
    \end{array}
    \right.$\\
  When $t$ is even,\\
  $f(\textbf{v})=
    \left\{\begin{array}{ll}
    3 & \text{if $\textbf{v}\in D_{n,t}^*$}   \\
    2 & \text{if $\textbf{v}=1^t$}     \\
    0 & \text{if $\textbf{v}\not\in D_{n,t}$}
    \end{array}
    \right.$
  \end{tabular}
\end{tabular}

\bigskip

{\bf Table 1:} Domination in the Sierpi\'{n}ski graphs $S(K_n,t).$
\end{center}

\bigskip

\section*{Acknowledgments}
This research is supported by Xiamen University Malaysia Research Fund under the project XMUMRF/2020-C5/IMAT/0015.

\end{document}